\documentclass[12pt]{aptpub}

\usepackage{amsmath,amstext,url}
\usepackage{color}

\renewcommand{\shorttitle}
{ \underline{{The down/up crossing properties of weighted Markov branching processes}}}
\numberwithin{equation}{section} % uncomment this to get equation numbering
\makeatletter \@addtoreset{equation}{section}

\makeatletter \@addtoreset{lemma}{section}

\makeatletter \@addtoreset{theorem}{section}

\makeatletter \@addtoreset{proposition}{section}

\makeatletter \@addtoreset{corollary}{section}

\makeatletter \@addtoreset{remark}{section}

\makeatletter \@addtoreset{definition}{section}

\makeatletter \@addtoreset{example}{section}

\begin{document}

\thispagestyle{firstpg}

\vspace*{1.5pc} \noindent \normalsize\textbf{\Large {The down/up crossing properties of weighted Markov branching processes}} \hfill

\vspace{12pt} \hspace*{0.75pc}{\small\textrm{\uppercase{Yanyun Li
}}}\hspace{-2pt}$^{*}$, {\small\textit{Central South University }}

\hspace*{0.75pc}{\small\textrm{\uppercase{Junping Li}}}
\hspace{-2pt}$^{**}$, {\small\textit{Central South University }}

\hspace*{0.75pc}{\small\textrm{\uppercase{Anyue Chen}}}
\hspace{-2pt}$^{***}$, {\small\textit{Southern University of Science and Technology}}

%\hspace*{0.75pc}{\small\textrm{\uppercase{Kai Wang Ng
% }}}\hspace{-2pt}$^{****}$, {\small\textit{University of Hong Kong }}

\par
\footnote{\hspace*{-0.75pc}$^{*}\,$Postal address:
 School of Mathematics and Statistics, Central
South University, Changsha, 410083, China. E-mail address:
2310962168@qq.com}

\par
\footnote{\hspace*{-0.75pc}$^{**}\,$Postal
address: School of Mathematics and Statistics, Central
South University, Changsha, 410083, China. E-mail address:
jpli@mail.csu.edu.cn}

\par
\footnote{\hspace*{-0.75pc}$^{***}\,$Postal
address: Department of Mathematics, Southern
University of Science and Technology, Shenzhen, 518055, China. E-mail address:
chenay@sustc.edu.cn}
\par
\renewenvironment{abstract}{%
\vspace{8pt} \vspace{0.1pc} \hspace*{0.25pc}
\begin{minipage}{14cm}
\footnotesize
{\bf Abstract}\\[1ex]
\hspace*{0.5pc}} {\end{minipage}}
\begin{abstract}
  We consider the weighted Markov branching processes which stops at state $0$. The joint probability distribution of fixed range crossing numbers of such processes is obtained by using a new method. In particular, the probability distribution of total death number is given for Markov branching process and the joint probability distributions of the total number of customers who ever been served and the total number of customers who ever entered the system are also given for bulk-arrival queueing systems.
\end{abstract}

\vspace*{12pt}
%\hspace*{2.25pc}
\parbox[b]{26.75pc}{{%\uppercase
}}
\par
\normalsize

\renewcommand{\amsprimary}[1]{
\vspace*{8pt}
\hspace*{2.25pc}
\parbox[b]{20.75pc}{\scriptsize
AMS 2000 Subject Classification: Primary 60J27 Secondary 60J35
     {\uppercase{#1}}}\par\normalsize}
\renewcommand{\ams}[2]{
\vspace*{8pt}
%\hspace*{2.25pc}
\parbox[b]{24.75pc}{\scriptsize
     AMS 2000 SUBJECT CLASSIFICATION: PRIMARY
     {\uppercase{#1}}\\ \phantom{
     AMS 2000
     SUBJECT CLASSIFICATION:
     }
SECONDARY
 {\uppercase{#2}}}\par\normalsize}

\ams{60J27}{60J35}
   % insert the primary ams number(s) in the first bracket
  % and the secondary ams number(s) in the second bracket
  %  e.g. \ams{60E20}{49G03;49F10}

\par
\vspace{5mm}
 \setcounter{section}{1}
 \setcounter{equation}{0}
 \setcounter{theorem}{0}
 \setcounter{lemma}{0}
 \setcounter{corollary}{0}
\noindent {\large \bf 1. Introduction}
\vspace{3mm}
\par
The ordinary Markov branching processes (MBPs) play an important role in the
classical field of stochastic processes. Some related references are Harris\cite{Harr63}, Athreya and Ney\cite{AKN72}, Asmussen and Hering\cite{ASHH83}.
\par
The basic property governing the evolution of an MBP is the branching property, different particles act independently when giving birth or death. In most realistic situation, however, the independence property is unlikely to be realized. Indeed, particles usually interact with each other. This is the main reason why there always have been an increasing effort to generalise the ordinary branching processes to more general branching models (see, for instance, Athreya and Jagers \cite{AKJP97}). Models like this were first studied by Sevast'yanov \cite{SBA49}. Some authors, including Vatutin \cite{VVA74}, Li $\&$ Chen\cite{LJCA06} and Li, Chen $\&$ Pakes \cite{LCP12} considered the branching process with state-independent immigration. Moreover, Li $\&$ Liu \cite{LJL11} added state-independent migration to the above branching process. Yamazato \cite{YM75} investigated a branching process with immigration which only occurs at state zero. Being viewed as an extension of Yamazato's model, Chen~\cite{RC97} discussed a more general branching process with or without resurrection. For the further discussion of this model, see Chen \cite{CA02} \cite{CA2002}, Chen, Li $\&$ Ramesh\cite{CLR2005} and Chen, Pollett, Zhang $\&$ Li \cite{CPLZ07} considered weighted Markov branching process. Within this structure, Chen, Li and Ramesh \cite{CLR2005} considered the uniqueness and extinction of weighted Markov branching processes, which is the further consideration of branching models discussed in Chen\cite{CA2002}.
\par
It is well-known that $0$ is the absorbing state for a branching system. Each particle in the system lives a random long time and gives a random number of new particles at its death time. The system stops when there is no particle in it. For such processes, it is interesting and also important to discuss the total number of particles who ever lived in the system (the total death number) or the total number of times that a particle in the system gave $m$ new particles at its death time (here $m\neq 0$ is fixed). For convenience, such number is called fixed range crossing number henceforth. In particular, the $-1$-range crossing number (down crossing number) is just the total death number for the process until its extinction and if $m>0$, then the $m$-range crossing (up crossing number) is just the total number of times that a particle in the system gave $m$ new particles at its death time. However, such problems are not considered in current references and still remain open. Since the down/up crossing numbers are random variables, therefore, it needs to discuss the probability distribution of $m$-range crossing number for the process until its extinction. The main purpose of this paper is to consider such problems for weighted Markov branching processes.
\par
In order to begin our discussion, we first define our model by specifying the infinitesimal generator, i.e., the so-called $Q$-matrix. Throughout this paper, let $\mathbb{Z}_+=\{0,1,2,\cdots \}$.
\par
\noindent
\begin{definition}\label{def1.1}\
A $Q$-matrix $Q=(q_{ij};\ i,j\in \mathbb{Z}_+)$ is called a weighted branching $Q$-matrix ( henceforth referred to as a WMB $Q$-matrix), if
\begin{eqnarray} \label{eq1.1}
 q_{ij}=\begin{cases}w_ib_{j-i+1},\ if \ i\geq 1, j\geq i-1, \\
                0 , \ \ \ \ \ \mbox{otherwise},
 \end{cases}
 \end{eqnarray}
 where
\begin{eqnarray} \label{eq1.2}
           b_j\geq 0 \ (j\neq 1),\  0<-b_1=\sum_{j\neq 1}b_j<\infty,\ w_i>0\ (i\geq 1).
\end{eqnarray}
\end{definition}
\par
\begin{definition}\label{def1.2}\
A weighted Markov branching process ( henceforth referred to as a WMBP) is a continuous-time Markov chain with state space $\mathbb{Z}_+$ whose transition function $P(t)=(p_{ij}(t);$ $i,j\in \mathbb{Z}_+)$ satisfies
\begin{eqnarray}\label{eq1.3}
p'_{ij}(t)=\sum_{k=0}^{\infty}p_{ik}(t)q_{kj},\ i\geq0,\ j\geq1,\
t\geq 0,
\end{eqnarray}
where $Q=(q_{ij};\ i,j\in \mathbb{Z}_+)$ is defined in (\ref{eq1.1})-(\ref{eq1.2}).
\end{definition}
\par
Chen, Li $\&$ Ramesh~\cite{CLR2005} derived the regularity conditions for WMBPs in terms of the death
rate $b_0$ and birth rates $\{b_k; k\geq 2\}$. Therefore, we assume the process is regular in the following.
\par
\vspace{5mm}
 \setcounter{section}{2}
 \setcounter{equation}{0}
 \setcounter{theorem}{0}
 \setcounter{lemma}{0}
 \setcounter{definition}{0}
 \setcounter{corollary}{0}
\noindent {\large \bf 2. Preliminaries}
 \vspace{3mm}
\par
Let $\mathbb{N}\subset \mathbb{Z}_+$ be a finite subset with $1\notin \mathbb{N}$ and $b_k>0$ for all $k\in \mathbb{N}$. The number of elements in $\mathbb{N}$ is denoted by $N$, i.e., $N=|\mathbb{N}|$. We will consider the $(\mathbb{N}-1)$-range crossing number of the process until its extinction, i.e., the joint probability distribution of the $N$-dimensional random vector $(N_i;\ i\in \mathbb{N})$, where $N_i$ is the $(i-1)$-range crossing number of the process until its extinction.
\par
In order to begin our discussion, define
\begin{eqnarray}\label{eq2.1}
  B(u)=\sum_{j=0}^{\infty}b_ju^j
\end{eqnarray}
and
\begin{eqnarray}\label{eq2.2}
B_{\mathbb{N}}(u,\emph{\textbf{v}})=\sum_{j\in \mathbb{N}}b_jv_j u^j,\ \ \ \bar{B}_{\mathbb{N}}(u)=\sum_{j\in \mathbb{N}^c}b_ju^j,
\end{eqnarray}
where $\emph{\textbf{v}}=(v_j;\ j\in \mathbb{N})$. It is obvious that $B(u),\ \bar{B}_{\mathbb{N}}(u)$ are well defined at least on $[0,1]$, and $B_{\mathbb{N}}(u,\emph{\textbf{v}})$ is well defined at least on $[0,1]^{N+1}$.
\par
The following lemma is due to mathematical analysis and thus the proof is omitted here.
\par
\begin{lemma} \label{le2.2}
Suppose that $\{f_{\textbf{k}};\ \textbf{k}\in \mathbb{Z}_+^N\}$ is a sequence on $\mathbb{Z}_+^N$, $F(\textbf{v})=\sum_{\textbf{k}\in \mathbb{Z}_+^N}\textbf{v}^{\textbf{k}}$ is the generating function of $\{f_{\textbf{k}};\ \textbf{k}\in \mathbb{Z}_+^N\}$. Then for any $j\in \mathbb{Z}_+$,
\begin{eqnarray}\label{eq2.6}
F^j(\textbf{v})=\sum_{\textbf{l}\in \mathbb{Z}_+^N}f^{*(j)}_{\textbf{l}}\textbf{v}^{\textbf{l}}
\end{eqnarray}
where
\begin{eqnarray*}
f^{*(0)}_{\textbf{0}}=1,\ f^{*(0)}_{\textbf{l}}=0\ (\textbf{l}\neq \textbf{0}),\ \ f^{*(j)}_{\textbf{l}}=\sum_{\textbf{k}^{(1)}+\cdots+\textbf{k}^{(j)}=\textbf{l} }f_{\textbf{k}^{(1)}}\cdots f_{\textbf{k}^{(j)}},\ (j\geq 1)
\end{eqnarray*}
is the $j$'th convolution of $\{f_{\textbf{k}};\ \textbf{k}\in \mathbb{Z}_+^N\}$.
\end{lemma}
\par
The function $\bar{B}_{\mathbb{N}}(u)+B_{\mathbb{N}}(u,\emph{\textbf{v}})$ will play an important role in our discussion. The following theorem reveals its properties.
\par
\begin{theorem}\label{th2.1}
{\rm{(i)}}\ For any $\textbf{v}\in[0,1]^{N+1}$,
\begin{eqnarray}\label{eq2.7}
\bar{B}_{\mathbb{N}}(u)+B_{\mathbb{N}}(u,\textbf{v})=0
\end{eqnarray}
 possesses at most $2$ roots in $[0,1]$. The minimal nonnegative root of $\bar{B}_{\mathbb{N}}(u)+B_{\mathbb{N}}(u,\textbf{v})=0$ is denoted by $\rho(\textbf{v})$, then $\rho(\textbf{v})\leq \rho$, where $\rho$ is the minimal nonnegative root of $B(u)=0$.
\par
{\rm{(ii)}}\ $\rho(\textbf{v})\in C^{\infty}([0,1)^{N})$ and $\rho(\textbf{v})$ can be expanded as a multivariate Taylor sieris
\begin{eqnarray}\label{eq2.8}
\rho(\textbf{v})=\sum_{\textbf{k}\in \mathbb{Z}_+^N}\rho_{\textbf{k}}\textbf{v}^{\textbf{k}}.
\end{eqnarray}
where $\rho_{\textbf{k}}\geq 0, \forall\ \textbf{k}\in \mathbb{Z}_+^N$.
\end{theorem}
\begin{proof}
 Note that $0\leq B_{\mathbb{N}}(u,\emph{\textbf{0}})\leq B_{\mathbb{N}}(u,\emph{\textbf{v}})\leq B_{\mathbb{N}}(u,\emph{\textbf{1}})$, we know that
\begin{eqnarray*}
\bar{B}_{\mathbb{N}}(u)+B_{\mathbb{N}}(u,\emph{\textbf{v}})\leq B(u).
\end{eqnarray*}
{\rm{(i)}} follows from Li and Chen \cite{Li-Chen08}. Next to prove {\rm{(ii)}}. Without loss of generality, suppose that $\mathbb{N}=\{2,3\}$ and thus
\begin{eqnarray*}
B_{\mathbb{N}}(u,\emph{\textbf{v}})=B_{\mathbb{N}}(u,v_2,v_3).
\end{eqnarray*}
Denote
$$
f(u,v_2,v_3)=B_{\mathbb{N}}(u,v_2,v_3)\ \ \ \rm{and}\ \ g(u)=\bar{B}_{\mathbb{N}}(u).
$$
Then,
$$
g(\rho(v_2,v_3))+f(\rho(v_2,v_3),v_2,v_3)\equiv 0
$$
and hence
\begin{eqnarray*}
\begin{cases}
g'_u(\rho(v_2,v_3))\cdot \rho'_{v_2}(v_2,v_3)+f'_u(\rho(v_2,v_3),v_2,v_3)\cdot\rho'_{v_2}(v_2,v_3)+f'_{v_2}(\rho(v_2,v_3),v_2,v_3)\equiv 0\\
g'_u(\rho(v_2,v_3))\cdot \rho'_{v_3}(v_2,v_3)+f'_u(\rho(v_2,v_3),v_2,v_3)\cdot\rho'_{v_3}(v_2,v_3)+f'_{v_3}(\rho(v_2,v_3),v_2,v_3)\equiv 0,
\end{cases}
\end{eqnarray*}
which implies that $\rho'_{v_2}$ and $\rho'_{v_3}$ are well defined in $[0,1)^2$ since
\begin{eqnarray*}
g'_u(\rho(v_2,v_3))+f'_u(\rho(v_2,v_3),v_2,v_3)< g'_u(\rho)+f'_u(\rho,1,1)=B'(\rho)\leq 0
\end{eqnarray*}
By recursion method, we know that $\rho(v_2,v_3)\in C^{\infty}([0,1)^2)$.
\par
Suppose that
\begin{eqnarray*}
\rho(\emph{\textbf{v}})=\sum_{\emph{\textbf{k}}\in \mathbb{Z}_+^N}\rho_{\emph{\textbf{k}}}\emph{\textbf{v}}^{\emph{\textbf{k}}}.
\end{eqnarray*}
\par
Substituting the above expression of $\rho(\emph{\textbf{v}})$ into (\ref{eq2.7}) yields
\begin{eqnarray*}
0& \equiv & \bar{B}_{\mathbb{N}}(\rho(\emph{\textbf{v}}))+B_{\mathbb{N}}(\rho(\emph{\textbf{v}}),\emph{\textbf{v}})\\
&=& \sum_{j\in \mathbb{N}^c}b_j(\rho(\emph{\textbf{v}}))^j+\sum_{j\in \mathbb{N}}b_j(\rho(\emph{\textbf{v}}))^jv_j\\
&=& \sum_{j\in \mathbb{N}^c}b_j\sum_{\emph{\textbf{l}}\geq \emph{\textbf{0}}}\rho^{*(j)}_{\emph{\textbf{l}}}
\emph{\textbf{v}}^{\emph{\textbf{l}}}+\sum_{j\in \mathbb{N}}b_j\sum_{\emph{\textbf{l}}\geq \emph{\textbf{0}}}\rho^{*(j)}_{\emph{\textbf{l}}}
\emph{\textbf{v}}^{\emph{\textbf{l}}+\emph{\textbf{e}}_j}\\
&=& \sum_{\emph{\textbf{l}}\geq \emph{\textbf{0}}}(\sum_{j\in \mathbb{N}^c}b_j\rho^{*(j)}_{\emph{\textbf{l}}})
\emph{\textbf{v}}^{\emph{\textbf{l}}}+\sum_{j\in \mathbb{N}}b_j\sum_{\emph{\textbf{l}}\geq \emph{\textbf{0}}}\rho^{*(j)}_{\emph{\textbf{l}}}
\emph{\textbf{v}}^{\emph{\textbf{l}}+\emph{\textbf{e}}_j}.
\end{eqnarray*}
\par
We next prove $\rho_{\emph{\textbf{l}}}\geq 0$ using mathematical induction respect to $\emph{\textbf{l}}\cdot \emph{\textbf{1}}$. If $\emph{\textbf{l}}\cdot \emph{\textbf{1}}=0$, then $\rho_{\emph{\textbf{0}}}=\rho(\emph{\textbf{0}})\geq 0$ since it is the minimal nonnegative root of $\bar{B}_{\mathbb{N}}(u)+B_{\mathbb{N}}(u,\emph{\textbf{0}})=0$. If $\emph{\textbf{l}}\cdot \emph{\textbf{1}}=1$, i.e., $\emph{\textbf{l}}=\emph{\textbf{e}}_k$ for some $k\in \mathbb{N}$.
Then,
\begin{eqnarray*}
\sum_{j\in \mathbb{N}^c}b_j\rho^{*(j)}_{\emph{\textbf{e}}_k}+b_k\rho^{*(k)}_{\emph{\textbf{0}}}
=0,
\end{eqnarray*}
i.e.,
\begin{eqnarray*}
\sum_{j\in \mathbb{N}^c}b_jj\rho^{j-1}_{\emph{\textbf{0}}}\rho_{\emph{\textbf{e}}_k}+b_k\rho^{k}_{\emph{\textbf{0}}}
=0.
\end{eqnarray*}
Hence
\begin{eqnarray*}
\rho_{\emph{\textbf{e}}_k}
=-\frac{b_k\rho^{k}_{\emph{\textbf{0}}}}{\bar{B}'_{\mathbb{N}}(\rho_{\emph{\textbf{0}}})}\geq 0,\ \ k\in \mathbb{N},
\end{eqnarray*}
since $\bar{B}'_{\mathbb{N}}(\rho_{\emph{\textbf{0}}})<0$.
\par
Assume $\rho_{\emph{\textbf{l}}}\geq 0$ for $\emph{\textbf{l}}$ satisfying $\emph{\textbf{l}}\cdot \emph{\textbf{1}}\leq m$, then for $\bar{\emph{\textbf{l}}}\cdot \emph{\textbf{1}}=m+1$, there exists $\emph{\textbf{l}}$ and $k\in \mathbb{N}$ such that $\bar{\emph{\textbf{l}}}=\emph{\textbf{l}}+\emph{\textbf{e}}_k$ and $\emph{\textbf{l}}\cdot \emph{\textbf{1}}\leq m$, therefore,
\begin{eqnarray*}
\sum_{j\in \mathbb{N}^c}b_j\rho^{*(j)}_{\emph{\textbf{l}}+\emph{\textbf{e}}_k}
+b_k\rho^{*(k)}_{\emph{\textbf{l}}}=0,
\end{eqnarray*}
i.e.,
\begin{eqnarray*}
\sum_{j\in \mathbb{N}^c}b_jj\rho^{j-1}_{\emph{\textbf{0}}}\rho_{\emph{\textbf{l}}+\emph{\textbf{e}}_k}
+\sum_{j\in \mathbb{N}^c\setminus\{1\}}b_j\sum_{\emph{\textbf{l}}^{(1)}+\cdots +\emph{\textbf{l}}^{(j)}=\emph{\textbf{l}}+\emph{\textbf{e}}_k,\ \emph{\textbf{l}}^{(1)}\cdot\emph{\textbf{1}},\cdots ,\emph{\textbf{l}}^{(j)}\cdot \emph{\textbf{1}}\leq m}\rho_{\emph{\textbf{l}}^{(1)}}
\cdots \rho_{\emph{\textbf{l}}^{(j)}}+b_k\rho^{*(k)}_{\emph{\textbf{l}}}=0.
\end{eqnarray*}
Hence
\begin{eqnarray*}
\rho_{\bar{\emph{\textbf{l}}}}=\rho_{\emph{\textbf{l}}+\emph{\textbf{e}}_k}
=-\frac{\sum_{j\in \mathbb{N}^c\setminus\{1\}}b_j\sum_{\emph{\textbf{l}}^{(1)}+\cdots +\emph{\textbf{l}}^{(j)}=\emph{\textbf{l}}+\emph{\textbf{e}}_k,\ \emph{\textbf{l}}^{(1)}\cdot\emph{\textbf{1}},\cdots ,\emph{\textbf{l}}^{(j)}\cdot \emph{\textbf{1}}\leq m}\rho_{\emph{\textbf{l}}^{(1)}}
\cdots \rho_{\emph{\textbf{l}}^{(j)}}+b_k\rho^{*(k)}_{\emph{\textbf{l}}}}{\bar{B}'_{\mathbb{N}}(\rho_{\emph{\textbf{0}}})}\geq 0,
\end{eqnarray*}
since $\bar{B}'_{\mathbb{N}}(\rho_{\emph{\textbf{0}}})<0$.
By mathematical induction, we know that $\rho_{\emph{\textbf{l}}}\geq 0, \forall\ \emph{\textbf{l}}\in \mathbb{Z}_+^N$. The proof is complete. \hfill $\Box$
\end{proof}

\par
\vspace{5mm}
 \setcounter{section}{3}
 \setcounter{equation}{0}
 \setcounter{theorem}{0}
 \setcounter{lemma}{0}
 \setcounter{definition}{0}
 \setcounter{corollary}{0}
\noindent {\large \bf 3. Down/up crossing property}
 \vspace{3mm}
\par
In this section, we consider the down/up crossing properties of weighted Markov branching process.
\par
Let $\mathbb{N}\subset \mathbb{Z}_+$ be a finite subset with $1\notin \mathbb{N}$ and $b_k>0$ for all $k\in \mathbb{N}$. $N=|\mathbb{N}|$ denotes the number of elements in $\mathbb{N}$.
\par
 The main purpose of this paper is to count the $(\mathbb{N}-1)$-range crossing numbers. However, the WMBP itself can not reveal such crossing numbers directly. Therefore, we need to find a new method to discuss the property of such crossing numbers. For this purpose, we construct a new $Q$-matrix $\tilde{Q}=(q_{(i,\emph{\textbf{k}}), (j,\emph{\textbf{l}})};\ (i,\emph{\textbf{k}}),$ $ (j,\emph{\textbf{l}})\in \mathbb{Z}_+^{N+1} )$.
\begin{eqnarray}\label{eq3.1}
 q_{(i,\emph{\textbf{k}}), (j,\emph{\textbf{l}})}=\begin{cases}w_ib_{j-i+1},\ \ \ if \ i\geq 1, j-i+1\in \mathbb{N}^c, \emph{\textbf{l}}=\emph{\textbf{k}} \\
                   w_ib_{j-i+1}, \ \ \ if \ i\geq 1, j-i+1\in \mathbb{N}, \emph{\textbf{l}}=\emph{\textbf{k}}+\emph{\textbf{e}}_{j-i+1}\\
                0 , \ \ \ \ \ \ \ \ \ \ \ \ otherwise.
                     \end{cases}
\end{eqnarray}
\par
Therefore, $\tilde{Q}$ determines a $(N+1)$-dimensional Markov chain $(X(t), \emph{\textbf{Y}}(t))$, where $X(t)$ is the weighted Markov branching process, $\emph{\textbf{Y}}(t)=(Y_k(t);k\in \mathbb{N})$ (assume $Y_k(0)=0\ (k\in \mathbb{N})$) counts the $(\mathbb{N}-1)$-range crossing numbers until time $t$. In particular,
\par
(i)\ if $\mathbb{N}=\{0\}$ then $Y_0(t)$ counts the down crossing number (i.e., the death number) of $\{X(t): t\geq 0\}$ until time $t$.
\par
(ii)\ If $\mathbb{N}=\{m\}\ (m\geq 2)$, then $Y_m(t)$ counts the $(m-1)$-range up crossing number of $\{X(t): t\geq 0\}$ until time $t$.
\par
(iii)\ If $\mathbb{N}=\{0,m\}\ (m\geq 2)$, then $\emph{\textbf{Y}}(t)=(Y_0(t),Y_m(t))$ counts the death number and the $(m-1)$-range up corssing number of $\{X(t): t\geq 0\}$ until time $t$.
\par
Let $P(t)=(p_{(i,\emph{\textbf{k}}), (j,\emph{\textbf{l}})}(t); (i,\emph{\textbf{k}}), (j,\emph{\textbf{l}})\in \mathbb{Z}_+^{N+1})$ denote the transition probability of $(X(t),\emph{\textbf{Y}}(t))$.
\begin{lemma}\label{le3.1}
For $P(t)$, we have
\begin{eqnarray}\label{eq3.2}
\sum_{(j,\textbf{l})\in \mathbb{Z}_+^{N+1}}p'_{(i,\textbf{0}),
 (j,\textbf{l})}(t)u^j\textbf{v}^{\textbf{l}}=[\bar{B}_{\mathbb{N}}(u)+B_{\mathbb{N}}(u,\textbf{v})]\cdot\sum_{j\geq 1,\textbf{k}\in \mathbb{Z}_+^N}p_{(i,\textbf{0}),
 (j,\textbf{k})}(t)w_ju^{j-1}\textbf{v}^{\textbf{k}}
\end{eqnarray}
where $\bar{B}_{\mathbb{N}}(u),\ B_{\mathbb{N}}(u,\textbf{v})$ are defined in $(\ref{eq2.2})$, $\textbf{v}^{\textbf{l}}=\prod_{k\in \mathbb{N}}v_k^{l_k}$. Moreover,
\begin{eqnarray}\label{eq3.3}
 && \sum_{(j,\textbf{l})\in \mathbb{Z}_+^{N+1}}p_{(i,\textbf{0}),
 (j,\textbf{l})}(t)u^j\textbf{v}^{\textbf{l}}-u^i\nonumber\\
 &=& [\bar{B}_{\mathbb{N}}(u)+B_{\mathbb{N}}(u,\textbf{v})]\cdot\sum_{j\geq 1,\textbf{k}\in \mathbb{Z}_+^N}(\int_0^tp_{(i,\textbf{0}),
 (j,\textbf{k})}(s)ds)\cdot w_ju^{j-1}\textbf{v}^{\textbf{k}}.
\end{eqnarray}
\end{lemma}
\begin{proof}
It follows from Kolmogorov forward equation that
\begin{eqnarray*}
 p'_{(i,\textbf{0}),
 (j,\emph{\textbf{l}})}(t)=\sum_{n\geq 1,\emph{\textbf{k}}\in \mathbb{Z}_+^N}p_{(i,\textbf{0}),
 (n,\emph{\textbf{k}})}(t)\cdot q_{(n,\emph{\textbf{k}}),(j,\emph{\textbf{l}})},\ \ \forall t\geq 0,\ i\in \mathbb{Z}_+.
\end{eqnarray*}
Multiplying $u^j\emph{\textbf{v}}^{\emph{\textbf{l}}}$ on both sides of the above equality and then summing over $j$ and $\emph{\textbf{l}}$ yield that
\par
\begin{eqnarray*}
 && \sum_{(j,\emph{\textbf{l}})\in \mathbb{Z}_+^{N+1}}p'_{(i,\textbf{0}),
 (j,\emph{\textbf{l}})}(t)\cdot u^j\emph{\textbf{v}}^{\emph{\textbf{l}}}\\
 &=&\sum_{(j,\emph{\textbf{l}})\in \mathbb{Z}_+^{N+1}}\sum_{n\geq 1,\emph{\textbf{k}}\in \mathbb{Z}_+^N}p_{(i,\textbf{0}),
 (n,\emph{\textbf{k}})}(t)\cdot q_{(n,\emph{\textbf{k}}),(j,\emph{\textbf{l}})}
 \cdot u^j\emph{\textbf{v}}^{\emph{\textbf{l}}}\\
 &=&\sum_{n\geq 1,\emph{\textbf{k}}\in \mathbb{Z}_+^N}\sum_{(j,\emph{\textbf{l}})\in \mathbb{Z}_+^{N+1}}p_{(i,\textbf{0}),
 (n,\emph{\textbf{k}})}(t)\cdot q_{(n,\emph{\textbf{k}}),(j,\emph{\textbf{l}})}
 \cdot u^j\emph{\textbf{v}}^{\emph{\textbf{l}}}\\
  &=&\sum_{n\geq 1,\emph{\textbf{k}}\in \mathbb{Z}_+^N}[\sum_{j:j-n+1\in\mathbb{N}}p_{(i,\textbf{0}),
 (n,\emph{\textbf{k}})}(t)\cdot q_{(n,\emph{\textbf{k}}),(j,\emph{\textbf{k}}+\emph{\textbf{e}}_{j-n+1})}
 \cdot u^j\emph{\textbf{v}}^{\emph{\textbf{k}}
 +\emph{\textbf{e}}_{j-n+1}}\\
 && +\sum_{j:j-n+1\in\mathbb{N}^c}p_{(i,\textbf{0}),
 (n,\emph{\textbf{k}})}(t)\cdot q_{(n,\emph{\textbf{k}}),(j,\emph{\textbf{k}})}
 \cdot u^j\emph{\textbf{v}}^{\emph{\textbf{k}}}]\\
 \end{eqnarray*}
 \begin{eqnarray*}
 &=&\sum_{n\geq 1,\emph{\textbf{k}}\in \mathbb{Z}_+^N}[\sum_{j:j-n+1\in\mathbb{N}}p_{(i,\textbf{0}),
 (n,\emph{\textbf{k}})}(t)\cdot w_nb_{j-n+1}
 \cdot u^j\emph{\textbf{v}}^{\emph{\textbf{k}}
 +\emph{\textbf{e}}_{j-n+1}}\\
 && +\sum_{j:j-n+1\in\mathbb{N}^c}p_{(i,\textbf{0}),
 (n,\emph{\textbf{k}})}(t)\cdot w_nb_{j-n+1}
 \cdot u^j\emph{\textbf{v}}^{\emph{\textbf{k}}}]\\
 &=&\sum_{n\geq 1,\emph{\textbf{k}}\in \mathbb{Z}_+^N}[B_{\mathbb{N}}(u,\emph{\textbf{v}})\cdot p_{(i,\textbf{0}),
 (n,\emph{\textbf{k}})}(t)\cdot w_n
 u^{n-1}\emph{\textbf{v}}^{\emph{\textbf{k}}}+\bar{B}_{\mathbb{N}}(u)\cdot p_{(i,\textbf{0}),
 (n,\emph{\textbf{k}})}(t)\cdot w_nu^{n-1}\emph{\textbf{v}}^{\emph{\textbf{k}}}]\\
 &=&[\bar{B}_{\mathbb{N}}(u)+B_{\mathbb{N}}(u,\emph{\textbf{v}})]\cdot\sum_{j\geq 1,\emph{\textbf{k}}\in \mathbb{Z}_+^N}p_{(i,\textbf{0}),
 (j,\emph{\textbf{k}})}(t)\cdot w_ju^{j-1}\emph{\textbf{v}}^{\emph{\textbf{k}}}.
\end{eqnarray*}
(\ref{eq3.2}) is proved. Integrating (\ref{eq3.2}) yields (\ref{eq3.3}). The proof is complete.
 \hfill $\Box$
\end{proof}
\par
Let
\begin{eqnarray}\label{eq3.4}
\tau=\inf\{t\geq 0;\ X(t)=0\}
\end{eqnarray}
be the extinction time of $\{X(t); t\geq 0\}$.
\par
The following theorem gives the joint probability generating function of $(\mathbb{N}-1)$-crossing numbers conditioned on $\tau<\infty$.
\par
\begin{theorem}\label{th3.1}
Suppose that $\{X(t); t\geq 0\}$ is a weighted Markov branching process with $X(0)=1$. Then
the probability generating function $G(\textbf{v})$ of $(\mathbb{N}-1)$-range crossing numbers conditioned on $\tau<\infty$ is given by
\begin{eqnarray}\label{eq3.5}
G(\textbf{v})=\rho^{-1}\sum_{\textbf{l}\in \mathbb{Z}_+^N}\rho_{\textbf{l}}\textbf{v}^{\textbf{l}},\ \ \textbf{v}\in [0,1]^N,
\end{eqnarray}
where $\rho$ is the minimal nonnegative root of $B(u)=0$, $\rho_{\textbf{0}}$ is the minimal nonnegative root of $\bar{B}_{\mathbb{N}}(u)=0$ and $\rho_{\textbf{l}}\ (\textbf{l}\neq \textbf{0})$ are given by the following recursion.
\begin{eqnarray}\label{eq3.6}
\rho_{\textbf{l}+\textbf{e}_k}
=-\frac{\sum_{j\in \mathbb{N}^c\setminus\{1\}}b_j\sum_{\textbf{l}^{(1)}+\cdots +\textbf{l}^{(j)}=\textbf{l}+\textbf{e}_k,\ \textbf{l}^{(1)}\cdot\textbf{1},\cdots ,\textbf{l}^{(j)}\cdot \textbf{1}\leq \textbf{l}\cdot \textbf{1}}\rho_{\textbf{l}^{(1)}}
\cdots \rho_{\textbf{l}^{(j)}}+b_k\rho^{*(k)}_{\textbf{l}}}{\bar{B}'_{\mathbb{N}}
(\rho_{\textbf{0}})},\ \ k\in \mathbb{N}.
\end{eqnarray}
\end{theorem}
\par
\begin{proof} Let $\tilde{Q}=(q_{(i,\emph{\textbf{k}}), (j,\emph{\textbf{l}})};\ (i,\emph{\textbf{k}}),$ $ (j,\emph{\textbf{l}})\in \mathbb{Z}_+^{N+1} )$ be the $Q$-matrix defined in (\ref{eq3.1}) and $(X(t), \emph{\textbf{Y}}(t))$ be the $\tilde{Q}$-process. Then $\emph{\textbf{Y}}(t)=(Y_k(t); k\in \mathbb{N})$ counts the $(\mathbb{N}-1)$-range crossing numbers until time $t$ and $\emph{\textbf{Y}}(\tau)=(Y_k(\tau); k\in \mathbb{N})$ counts the $(\mathbb{N}-1)$-range crossing numbers conditioned on $\tau<\infty$.
\par
Let $P(t)=(p_{(i,\emph{\textbf{k}}), (j,\emph{\textbf{l}})}(t); (i,\emph{\textbf{k}}), (j,\emph{\textbf{l}})\in \mathbb{Z}_+^{N+1})$ denote the transition probability of $(X(t),\emph{\textbf{Y}}(t))$. It follows from Lemma~\ref{le3.1} that
\begin{eqnarray}\label{eq3.7}
 && \sum_{(j,\emph{\textbf{l}})\in \mathbb{Z}_+^{N+1}}p_{(1,\emph{\textbf{0}}),
 (j,\emph{\textbf{l}})}(t)\cdot u^j\emph{\textbf{v}}^{\emph{\textbf{l}}}-u\nonumber\\
 &=& [\bar{B}_{\mathbb{N}}(u)+B_{\mathbb{N}}(u,\emph{\textbf{v}})]\cdot\sum_{j\geq 1,\emph{\textbf{k}}\in \mathbb{Z}_+^N}(\int_0^tp_{(1,\emph{\textbf{0}}),
 (j,\emph{\textbf{k}})}(s)ds)\cdot w_ju^{j-1}\emph{\textbf{v}}^{\emph{\textbf{k}}}.
\end{eqnarray}
Letting $t\rightarrow \infty$ on both sides of (\ref{eq3.7}) and noting that
$\lim_{t\rightarrow \infty}p_{(1,\emph{\textbf{0}}),(j,\emph{\textbf{l}})}(t)=0$ for $j\geq 1$ (since $(j,\emph{\textbf{l}})$ are transient states for $j\geq 1$) yield
\begin{eqnarray}\label{eq3.8}
\sum_{\emph{\textbf{l}}\in\mathbb{Z}_+^N}a_{\emph{\textbf{l}}}\emph{\textbf{v}}^{\emph{\textbf{l}}}-u
=[\bar{B}_{\mathbb{N}}(u)+B_{\mathbb{N}}(u,\emph{\textbf{v}})]\cdot\sum_{j\geq 1,\emph{\textbf{k}}\in \mathbb{Z}_+^N}(\int_0^{\infty}p_{(1,\emph{\textbf{0}}),
 (j,\emph{\textbf{k}})}(t)dt)\cdot w_ju^{j-1}\emph{\textbf{v}}^{\emph{\textbf{k}}}.
\end{eqnarray}
where $a_{\emph{\textbf{l}}}=\lim_{t\rightarrow \infty}p_{(1,\emph{\textbf{0}}),
 (0,\emph{\textbf{l}})}(t)$. By Theorem~\ref{th2.1}, let $u=\rho(\emph{\textbf{v}})$ in (\ref{eq3.8}), we obtain that
\begin{eqnarray*}
\sum_{\emph{\textbf{l}}\in\mathbb{Z}_+^N}a_{\emph{\textbf{l}}}\emph{\textbf{v}}^{\emph{\textbf{l}}}
-\rho(\emph{\textbf{v}})\equiv 0
\end{eqnarray*}
which implying $a_{\emph{\textbf{l}}}=\rho_{\emph{\textbf{l}}}\ (\emph{\textbf{l}}\in\mathbb{Z}_+^N)$.
\par
Finally,
\begin{eqnarray*}
&&P(\emph{\textbf{Y}}(\tau)=\emph{\textbf{l}}|\tau<\infty)\\
&=& \frac{P(\emph{\textbf{Y}}(\tau)=\emph{\textbf{l}}, \tau<\infty)}{P(\tau<\infty)}\\
&=& \rho^{-1}\lim_{t\rightarrow \infty}P(\emph{\textbf{Y}}(t)=\emph{\textbf{l}}, \tau<t)\\
&=& \rho^{-1}\lim_{t\rightarrow \infty}p_{(1,\emph{\textbf{0}}),
 (0,\emph{\textbf{l}})}(t)\\
 &=& \rho^{-1}\rho_{\emph{\textbf{l}}}.
\end{eqnarray*}
The proof is complete. \hfill $\Box$
\end{proof}
\par
\begin{remark}\label{re3.1}
Theorem~\ref{th3.1} gives the joint probability generating function of $(\mathbb{N}-1)$-crossing numbers conditioned on $\tau<\infty$. Therefore, the joint probability distribution of $(\mathbb{N}-1)$-crossing numbers $\emph{\textbf{Y}}(\tau)$ conditioned on $\tau<\infty$
\begin{eqnarray*}
P(\emph{\textbf{Y}}(\tau)=\emph{\textbf{l}}\ |\tau<\infty)=\rho^{-1}\rho_{\emph{\textbf{l}}},\ \ \emph{\textbf{l}}\in \mathbb{Z}_+^N.
\end{eqnarray*}
\par
If $X(t)$ starts from $X(0)=i (>1)$, then the joint probability generating function of $(\mathbb{N}-1)$-crossing numbers conditioned on $\tau<\infty$ is
\begin{eqnarray*}
G_i(\emph{\textbf{v}})=[G(\emph{\textbf{v}})]^i
\end{eqnarray*}
and the joint probability distribution of $(\mathbb{N}-1)$-crossing numbers $\emph{\textbf{Y}}(\tau)$ conditioned on $\tau<\infty$
\begin{eqnarray*}
P(\emph{\textbf{Y}}(\tau)=\emph{\textbf{l}}\ |\tau<\infty)=\rho^{-i}\cdot\rho^{*(i)}_{\emph{\textbf{l}}},\ \ \emph{\textbf{l}}\in \mathbb{Z}_+^N,
\end{eqnarray*}
where $\{\rho^{*(i)}_{\emph{\textbf{l}}};\ \emph{\textbf{l}}\in \mathbb{Z}_+^N\}$ is the $i$'th convolution of $\{\rho_{\emph{\textbf{l}}};\ \emph{\textbf{l}}\in \mathbb{Z}_+^N\}$.
\end{remark}
\par
As direct consequences of Theorem~\ref{th3.1} and Remark~\ref{re3.1}, the following corollaries give the probability distributions of death number and $(m-1)$-range up crossing number conditioned on $\tau<\infty$ for fixed $m>1$.
\par
\begin{corollary}\label{cor3.1}
Suppose that $\{X(t); t\geq 0\}$ is a weighted Markov branching process with $X(0)=1$. Then
the probability generating function $G(v)$ of death number conditioned on $\tau<\infty$ is given by
\begin{eqnarray*}
G(v)=\rho^{-1}\cdot \sum_{k=0}^{\infty}\rho_{k}v^{k}
\end{eqnarray*}
and hence the probability distribution of death number $Y_0(\tau)$ is given by
\begin{eqnarray*}
P(Y_0(\tau)=k\ |\tau<\infty)=\rho^{-1}\cdot\rho_{k},\ \ k\geq 0,
\end{eqnarray*}
where $\rho_{0}=0$ and $\rho_{k}\ (k\geq 1)$ are given by the following recursion.
\begin{eqnarray*}
&&\rho_1=-b_1^{-1}\cdot b_0,\nonumber\\
&& \rho_{k+1}
=-b_1^{-1}\cdot\sum_{j=2}^{k+1}b_j\rho^{*(j)}_{k+1},\ \ k\geq 1.
\end{eqnarray*}
\end{corollary}
\begin{proof}
Note that $\mathbb{N}=\{0\}$, $\rho_0=0$ is the minimal nonnegative root of
$$
\bar{B}_{\mathbb{N}}(u)=\sum_{j=1}^{\infty}b_ju^j=0
$$
and
\begin{eqnarray*}
\sum_{j=2}^{\infty}b_j\sum_{l_1+\cdots +l_j=k+1,\ l_1,\cdots ,l_j\leq k}\rho_{l_1}
\cdots \rho_{l_j}=\sum_{j=2}^{k+1}b_j\rho^{*(j)}_{k+1}.
\end{eqnarray*}
By Theorem~\ref{th3.1}, we immediately obtain the result. The proof is complete. \hfill $\Box$
\end{proof}
\par
\begin{remark}\label{re3.2}
By checking the proof of Theorem~\ref{th3.1}, one can find that the joint probability distribution of crossing numbers conditioned on $\tau<\infty$ does not depend on the sequence $\{w_i;\ i\geq 1\}$, therefore, we have
\par
\rm{(i)}\ For Markov branching process, the probability distribution of death number $Y_0(\tau)$ conditioned on $\tau<\infty$ is given by
\begin{eqnarray*}
P(Y_0(\tau)=k\ |\tau<\infty)=\rho^{-1}\cdot\rho_{k},\ \ k\geq 0
\end{eqnarray*}
where $\rho_k\ (k\geq 0)$ are given in Corollary~\ref{cor3.1}.
\par
\rm{(ii)}\ For $M^X/M/1$ queueing process, the probability distribution of service number $Y_0(\tau)$ in a busy period is given by
\begin{eqnarray*}
P(Y_0(\tau)=k\ |\tau<\infty)=\rho^{-1}\cdot\rho_{k},\ \ k\geq 0
\end{eqnarray*}
where $\rho_k\ (k\geq 0)$ are given in Corollary~\ref{cor3.1}.
\end{remark}
\par
\begin{corollary}\label{cor3.2}
Suppose that $\{X(t); t\geq 0\}$ is a weighted Markov branching process with $X(0)=1$ and $m(>1)$ is fixed. Then
the probability generating function $G(v)$ of $(m-1)$-range up-crossing number conditioned on $\tau<\infty$ is given by
\begin{eqnarray*}
G(v)=\rho^{-1}\cdot\sum_{k=0}^{\infty}\rho_{k}v^{k}
\end{eqnarray*}
and hence the probability distribution of $(m-1)$-range up-crossing number $Y_m(\tau)$ conditioned on $\tau<\infty$ is given by
\begin{eqnarray*}
P(Y_m(\tau)=k\ |\tau<\infty)=\rho^{-1}\cdot\rho_{k},\ \ k\geq 0,
\end{eqnarray*}
where $\rho_{0}$ is the minimal nonnegative root of $B_m(u)=\sum_{j\neq m}b_ju^j=0$ and $\rho_{k}\ (k\geq 1)$ are given by the following recursion.
\begin{eqnarray*}
&&\rho_1=-\frac{b_m\rho_0^m}{B'_m(\rho_0)},\nonumber\\
&& \rho_{k+1}=-\frac{\sum_{i\neq 1,m}b_i\sum_{j_1+\cdots+j_i=k+1;\ j_1,\cdots,j_i\leq k}\rho_{j_1}\cdots\rho_{j_i}+b_m\rho^{*(m)}_{k}}{B'_m(\rho_0)},\ \ k\geq 1,
\end{eqnarray*}
where $\{\rho^{*(m)}_k;\ k\in \mathbb{Z}_+\}$ is the $m$'th convolution of $\{\rho_k;\ k\in \mathbb{Z}_+\}$.
\end{corollary}
\begin{proof}
It follows directly from Theorem~\ref{th3.1} with $\mathbb{N}=\{m\}$. \hfill $\Box$
\end{proof}
\par
\begin{theorem}\label{th3.2}
Suppose that $\{X(t); t\geq 0\}$ is a weighted Markov branching process with $X(0)=1$, $\rho=1$ \rm{(i.e., $B'(1)\leq 0$)}, $m\neq 1$, $Y_m(\tau)$ is the $(m-1)$-range crossing number. Then
\begin{eqnarray*}
  E[Y_m(\tau)]=\sum_{k=0}^{\infty}k\rho_k
\end{eqnarray*}
 and
 \begin{eqnarray*}
    Var[(Y_m(\tau)]=\sum_{k=0}^{\infty}k^2\rho_k-(\sum_{k=0}^{\infty}k\rho_k)^2,
 \end{eqnarray*}
 where $\rho_k (k\geq 0)$ are given in Corollary~\ref{cor3.1} \rm{(}in the case $m=0$\rm{)} and in Corollary~\ref{cor3.2} \rm{(}in the case $m>1$\rm{)}.
\end{theorem}
\begin{proof}
Since $\rho=1$, we know that $G(v)$ is the probability generating function of $Y_m(\tau)$, therefore,
\begin{eqnarray*}
  E[Y_m(\tau)]=G'(1^-), \ \  Var[(Y_m(\tau)]= G^{''}{(1^-)}+G'(1^-)-[G'(1^-)]^2.
\end{eqnarray*}
The results follow from Corollaries~\ref{cor3.1} and~\ref{cor3.2}. The proof is complete.\hfill $\Box$
\end{proof}
\par
Theorem~\ref{th3.1} gives the joint probability distribution of $(\mathbb{N}-1)$-range crossing numbers conditioned on $\tau<\infty$. We now turn to consider the case $\tau=\infty$. For simplicity, we only consider the case that $\mathbb{N}$ contains a single element, the general case is similar.
\par
Let $m\in \mathbb{Z}_+$ with $b_m>0$ and $\tilde{Q}_m=(q_{(i,k), (j,l)};\ (i,k), (j,l)\in \mathbb{Z}_+^2)$ be a $Q$-matrix defined by (\ref{eq3.1}) with $\mathbb{N}=\{m\}$. Suppose that $(X(t), Y_m(t))$ is the $\tilde{Q}_m$-process, where $X(t)$ is the weighted Markov branching process, $Y_m(t)$ counts the $(m-1)$-range crossing number until time $t$. $P(t)=(p_{(i,k), (j,l)};\ (i,k), (j,l)\in \mathbb{Z}_+^2)$ is the $\tilde{Q}_m$-function.
\par
\begin{theorem}\label{th3.3}
Suppose that $(X(t), Y_m(t))$ is the $\tilde{Q}_m$-process with $(X(0),Y(0))=(1,0)$ and $\rho<1$.  Then
\begin{eqnarray}\label{eq3.9}
P(Y_m(\infty)=\infty|\tau=\infty)=1.
\end{eqnarray}
\end{theorem}
\par
\begin{proof}
It follows from Lemma~\ref{le3.1} with $\mathbb{N}=\{m\}$ that for any $u,v\in[0,1]$,
\begin{eqnarray*}
 && \sum_{(j,k)\in \mathbb{Z}_+^{2}}p_{(1,0),
 (j,k)}(t)\cdot u^jv^k-u=B_m(u,v)\cdot\sum_{k=0}^{\infty}\sum_{j=1}^{\infty}(\int_0^tp_{(1,0),
 (j,k)}(s)ds)\cdot w_ju^{j-1}v^k\ \ \ \ \ \ \ \
\end{eqnarray*}
where $B_m(u,v)=\sum_{i\neq m}b_iu^i+b_{m}u^mv$.
i.e.,
\begin{eqnarray}\label{eq3.10}
 && \sum_{k=0}^{\infty}p_{(1,0),
 (0,k)}(t)\cdot v^k+\sum_{k=0}^{\infty}\sum_{j=1}^{\infty}p_{(1,0),
 (j,k)}(t)\cdot u^jv^k-u\nonumber\\
 &=& B_m(u,v)\cdot\sum_{k=0}^{\infty}\sum_{j=1}^{\infty}(\int_0^tp_{(1,0),
 (j,k)}(s)ds)\cdot w_ju^{j-1}v^k.
\end{eqnarray}
Letting $t\rightarrow \infty$ in the above equality and noting $\lim_{t\rightarrow \infty}\sum_{k=0}^{\infty}p_{(1,0),
 (0,k)}(t)v^k=\rho(v)$ yield
\begin{eqnarray*}
\rho(v)-u= B_m(u,v)\cdot\sum_{k=0}^{\infty}\sum_{j=1}^{\infty}(\int_0^{\infty}p_{(1,0),
 (j,k)}(s)ds)\cdot w_ju^{j-1}v^k,\ \ \ \forall\ u,v\in [0,1)
\end{eqnarray*}
where $\rho(v)$, the minimal nonnegative root of $B_m(u,v)=0$ for fixed $v\in [0,1]$, is given in Corollary~\ref{cor3.2}. Letting $u\uparrow 1$ and using monotone convergence theorem yield
\begin{eqnarray*}
1-\rho(v)= b_m(1-v)\cdot\sum_{k=0}^{\infty}\sum_{j=1}^{\infty}(\int_0^{\infty}p_{(1,0),
 (j,k)}(s)ds)\cdot w_jv^k,\ \ \ \forall\ v\in [0,1)
\end{eqnarray*}
which implies
\begin{eqnarray}\label{eq3.11}
  \sum_{j=1}^{\infty}(\int_0^{\infty}p_{(1,0),(j,k)}(s)ds)\cdot w_j=b_m^{-1}\cdot (1-\sum_{i=0}^k\rho_i),\ \ \ k\geq 0
\end{eqnarray}
and
$$
\sum_{k=0}^{\infty}\sum_{j=1}^{\infty}(\int_0^{\infty}p_{(1,0),
 (j,k)}(s)ds)\cdot w_jv^k<\infty,\ \ \ \ \forall\ v\in [0,1).
$$
\par
On the other hand, letting $u\uparrow1$ in (\ref{eq3.10}) yields
 \begin{eqnarray*}
 && 1-\sum_{k=0}^{\infty}p_{(1,0),
 (0,k)}(t)\cdot v^k-\sum_{k=0}^{\infty}P(Y_m(t)=k, \tau>t)v^k\nonumber\\
 &=& b_m(1-v)\cdot\sum_{k=0}^{\infty}\sum_{j=1}^{\infty}(\int_0^tp_{(1,0),
 (j,k)}(s)ds)\cdot w_jv^k
\end{eqnarray*}
Therefore,
 \begin{eqnarray*}
&&  \sum_{j=1}^{\infty}(\int_0^tp_{(1,0),(j,k)}(s)ds)\cdot w_j=b_m^{-1}\cdot [1-\sum_{i=0}^kp_{(1,0),
 (0,i)}(t)-P(Y_m(t)\leq k, \tau>t)],\ \ \ k\geq 0\\
\end{eqnarray*}
Letting $t\rightarrow \infty$ and using monotone convergence theorem yield
 \begin{eqnarray*}
&& \sum_{j=1}^{\infty}(\int_0^{\infty}p_{(1,0),(j,k)}(s)ds)\cdot w_j=b_m^{-1}\cdot [1-\sum_{i=0}^k\rho_i-P(Y_m(\infty)\leq k, \tau=\infty)],\ \ \ k\geq 0.
\end{eqnarray*}
Comparing the above equality with (\ref{eq3.11}), we see that
 \begin{eqnarray*}
P(Y_m(\infty)\leq k, \tau=\infty)=0,\ \ \ k\geq 0.
\end{eqnarray*}
Hence,
 \begin{eqnarray*}
P(Y_m(\infty)=\infty, \tau=\infty)=1.
\end{eqnarray*}
The proof is complete. \hfill $\Box$
\end{proof}
\par
Finally, we give two examples to illustrate the conclusions obtained above.
\begin{example}\label{ex3.1}
Let $X(0)=1$ and
\begin{eqnarray*}
B(u)=\mu -(\mu+\lambda)u+\lambda u^2
\end{eqnarray*}
and $\mathbb{N}=\{0,2\}$. Then
\begin{eqnarray*}
B_{\mathbb{N}}(u,y,z)=\mu y+\lambda zu^2,\ \ \bar{B}_{\mathbb{N}}(u)=-(\mu+\lambda)u.
\end{eqnarray*}
Consider
\begin{eqnarray}\label{eq3.12}
\bar{B}_{\mathbb{N}}(u)+B_{\mathbb{N}}(u,y,z)=\mu y-(\mu+\lambda)u+\lambda zu^2=0.
\end{eqnarray}
The minimal nonnegative root of (\ref{eq3.12}) is
\begin{eqnarray*}
\rho(y,z)=\frac{\mu+\lambda-\sqrt{(\mu+\lambda)^2-4\mu \lambda zy}}{2\lambda z},\ \ \rho=\rho(1,1)=\frac{\mu}{\lambda}\wedge 1.
\end{eqnarray*}
Using Tailor series $\sqrt{1+x}=1+\frac{1}{2}x+\sum_{n=2}^{\infty}\frac{(2n-3)!!}{2^nn!}(-1)^{n-1}x^n$
yields
\begin{eqnarray*}
\rho(y,z)&=&\frac{\mu+\lambda}{2\lambda z}\cdot\left[1-\sqrt{1-\frac{4\mu \lambda zy}{(\mu+\lambda)^2}}\right]\\
&=& \frac{\mu+\lambda}{2\lambda z}\cdot\left[1-(1-\frac{4\mu \lambda zy}{2(\mu+\lambda)^2}+\sum_{n=2}^{\infty}\frac{(2n-3)!!}{2^nn!}(-1)^{n-1}(-\frac{4\mu \lambda zy}{(\mu+\lambda)^2})^n)\right]\\
&=& \frac{\mu+\lambda}{2\lambda z}\cdot\left[\frac{4\mu \lambda zy}{2(\mu+\lambda)^2}+\sum_{n=2}^{\infty}\frac{(2n-3)!!}{2^nn!}(\frac{4\mu \lambda zy}{(\mu+\lambda)^2})^n\right]\\
&=& \frac{\mu y}{\mu+\lambda}+\sum_{n=2}^{\infty}
\frac{(2n-3)!!2^{n-1}\mu^{n}\lambda^{n-1}}{n!(\mu+\lambda)^{2n-1}}
z^{n-1}y^n.\\
\end{eqnarray*}
Therefore,
\begin{eqnarray*}
\begin{cases}
\rho_{1,0}=\frac{\mu}{\mu+\lambda},\\
\rho_{n,n-1}=\frac{(2n-3)!!2^{n-1}\mu^{n}\lambda^{n-1}}{n!(\mu+\lambda)^{2n-1}},\ \ n\geq 2\\
\rho_{n,m}=0,\ \ \rm{otherwise}
\end{cases}
\end{eqnarray*}
and hence,
\par
\rm{(i)}\ if $\mu \geq \lambda$, then
\begin{eqnarray*}
\begin{cases}
P((Y_0(\tau),Y_2(\tau))=(1,0))=\frac{\mu}{\mu+\lambda},\\
P((Y_0(\tau),Y_2(\tau))=(n,n-1))=\frac{(2n-3)!!2^{n-1}\mu^{n}\lambda^{n-1}}{n!(\mu+\lambda)^{2n-1}},\ \ n\geq 2\\
P((Y_0(\tau),Y_2(\tau))=(n,m))=0,\ \ \rm{otherwise},
\end{cases}
\end{eqnarray*}
\begin{eqnarray*}
\begin{cases}
P(Y_0(\tau)=0)=0,\\
P(Y_0(\tau)=1)=\frac{\mu}{\mu+\lambda},\\
P(Y_0(\tau)=n)=\frac{(2n-3)!!2^{n-1}\mu^{n}\lambda^{n-1}}{n!(\mu+\lambda)^{2n-1}},\ \ n\geq 2,
\end{cases}
\end{eqnarray*}
\begin{eqnarray*}
\begin{cases}
P(Y_2(\tau)=0)=\frac{\mu}{\mu+\lambda},\\
P(Y_2(\tau)=n)=\frac{(2n-1)!!2^{n}\mu^{n+1}\lambda^{n}}{(n+1)!(\mu+\lambda)^{2n+1}},\ \ n\geq 1.
\end{cases}
\end{eqnarray*}
\par
\rm{(ii)}\ if $\mu <\lambda$, then
\begin{eqnarray*}
\begin{cases}
P((Y_0(\tau),Y_2(\tau))=(1,0)\ |\tau<\infty)=\frac{\lambda}{\mu+\lambda},\\
P((Y_0(\tau),Y_2(\tau))=(n,n-1)\ |\tau<\infty)=\frac{(2n-3)!!2^{n-1}\mu^{n-1}\lambda^{n}}{n!(\mu+\lambda)^{2n-1}},\ \ n\geq 2\\
P((Y_0(\tau),Y_2(\tau))=(n,m)\ |\tau<\infty)=0,\ \ \rm{otherwise},
\end{cases}
\end{eqnarray*}
\begin{eqnarray*}
\begin{cases}
P(Y_0(\tau)=0\ |\tau<\infty)=0,\\
P(Y_0(\tau)=1\ |\tau<\infty)=\frac{\lambda}{\mu+\lambda},\\
P(Y_0(\tau)=n\ |\tau<\infty)=\frac{(2n-3)!!2^{n-1}\mu^{n-1}\lambda^{n}}{n!(\mu+\lambda)^{2n-1}},\ \ n\geq 2,
\end{cases}
\end{eqnarray*}
\begin{eqnarray*}
\begin{cases}
P(Y_2(\tau)=0\ |\tau<\infty)=\frac{\lambda}{\mu+\lambda},\\
P(Y_2(\tau)=n\ |\tau<\infty)=\frac{(2n-1)!!2^{n}\mu^{n}\lambda^{n+1}}{(n+1)!(\mu+\lambda)^{2n+1}},\ \ n\geq 1.
\end{cases}
\end{eqnarray*}
\end{example}
\par
\begin{example}\label{ex3.2}
Let $X(0)=1$ and
\begin{eqnarray*}
B(u)=2q-3pu+u^3,
\end{eqnarray*}
where $q>0$ and $3p=2q+1$. For $v\in [0,1)$, consider
\begin{eqnarray}\label{eq3.13}
B(u,v)=2qv-3pu+u^3.
\end{eqnarray}
Let $\rho(v)$ be the minimal nonnegative root of (\ref{eq3.13}) for fixed $v\in [0,1]$. Obviously, $\rho(0)=0$ and we can assume
$$
  \rho(v)=\sum_{k=1}^{\infty}\rho_kv^k.
$$
Then
\begin{eqnarray}\label{eq3.14}
   2qv-3p\rho(v)+\rho^3(v)\equiv 0, \ \ v\in [0,1).
\end{eqnarray}
It is easy to see that
\begin{eqnarray}\label{eq3.15}
 \rho_1=\frac{2q}{3p},\ \rho_2=0.
\end{eqnarray}
Differentiating (\ref{eq3.14}) yields
\begin{eqnarray*} -p\rho^{(n+1)}(v)+\sum_{k=0}^{n}C_n^k(\sum_{i=0}^kC_k^i\rho^{(i)}(v)\rho^{(k-i)}(v))\rho^{(n-k+1)}(v)\equiv 0,\ \ n\geq 2.
\end{eqnarray*}
Letting $v=0$ in the above equality yields
\begin{eqnarray*}
p(n+1)!\rho_{n+1}=\sum_{k=2}^{n}C_n^k(\sum_{i=1}^{k-1}C_k^ii!(k-i)!\rho_{i}\rho_{k-i})(n-k+1)!\rho_{n-k+1},\ \ n\geq 2.
\end{eqnarray*}
i.e.,
\begin{eqnarray}\label{eq3.16}
\rho_{n+1}=\frac{1}{p(n+1)}\cdot \sum_{k=2}^{n}(n-k+1)\rho_{n-k+1}\cdot(\sum_{i=1}^{k-1}\rho_{i}\rho_{k-i}),\ \ n\geq 2.
\end{eqnarray}
Therefore, if $p\geq 1$, then
\begin{eqnarray*}
\begin{cases}
P(Y_0(\tau)=0)=0,\\
P(Y_0(\tau)=1)=\frac{2q}{3p},\\
P(Y_0(\tau)=2)=0,\\
P(Y_0(\tau)=n)=\rho_n,\ \ n\geq 3.
\end{cases}
\end{eqnarray*}
 If $p<1$, then
\begin{eqnarray*}
\begin{cases}
P(Y_0(\tau)=0\ |\tau<\infty)=0,\\
P(Y_0(\tau)=1 |\tau<\infty)=\frac{2q}{3p\rho},\\
P(Y_0(\tau)=2|\tau<\infty)=0,\\
P(Y_0(\tau)=n|\tau<\infty)=\rho^{-1}\cdot\rho_n,\ \ n\geq 3
\end{cases}
\end{eqnarray*}
where $\rho=\rho(1)$ is the minimal nonnegative root of $B(u)=0$ and $\rho_n$ is given in (\ref{eq3.15}) and (\ref{eq3.16}).
\end{example}
\section*{Acknowledgement}
\par
 This work is substantially supported by the National Natural Sciences Foundations of China (No. 11771452, No. 11971486).


\begin{thebibliography}{10}

\bibitem{And91}
{\sc Anderson W.}(1991).
\newblock {\em Continuous-Time {M}arkov Chains: An Applications-Oriented
  Approach}.
\newblock Springer-Verlag, New York.

\bibitem{AKJP97}
{\sc Asmussen S. and Jagers P.(1997).}
\newblock {Classical and Mordern Branching Processes, Sptinger, Berlin.}


\bibitem{ASHH83}
{\sc Asmussen S. and Hering H.}(1983).
\newblock{\em Branching Processes.}
\newblock Birkhauser, Boston.

\bibitem{Ath94}
{\sc Athreya K.B.}(1994).
\newblock Large Deviation Rates for Branching Processes--I. Single Type Case.
\newblock {\em The Annals of Appl. Probab.}, 4(3):779-790.

\bibitem{AKN72}
{\sc Athreya K.B. and Ney P.E.}(1972).
\newblock {\em Branching Processes.}
\newblock {Springer, Berlin.}

\bibitem{CA02}
{\sc Chen A.Y.}(2002).
\newblock {Uniqueness and extinction properties of generalised Markov branching processes.}
\newblock {\em J. Math. Anal. Appl.}, 274(2):482-494

\bibitem{CA2002}
{\sc Chen A.Y.}(2002).
\newblock {Ergodicity and stability of generalised Markov branching processes with resurrection.}
\newblock {\em J. Appl. Probab.}, 39(4):786-803

\bibitem{CLR2005}
{\sc Chen A.Y. and Li J.P. and Ramesh N.}(2005).
\newblock {Uniqueness and extinction of weighted {M}arkov branching processes.}
\newblock {\em Methodol. Comput. Appl. Probab.}, 7(4):489-516


%\bibitem{ALR2000RANDOM}
%{\sc Chen, A. and Li, J. and Ramesh, N.}
%\newblock {Random Time Change in Weighted Markov Branching Processes.}


\bibitem{CPLZ07}
{\sc Chen A.Y. and Pollett P. and Li J.P. and Zhang H.J.}(2007).
\newblock {A remark on the uniqueness of weighted Markov branching processes.}
\newblock {\em J. Appl. Probab.}, 44(1):279-283


%\bibitem{CT88}
%{\sc Chow,Y.S. and Teicher H.} (1988).
%\newblock {\em  Probability Theory:Independence, Interchangeability, Martingales.}
%\newblock {Springer, Newyork.}

\bibitem{Harr63}
{\sc Harris T.E.}(1963).
\newblock {\em The theory of branching processes.}
\newblock {Springer, Berlin and Newyork.}


%\bibitem{KaZh83}
%{\sc Karp,R. and Zhang,Y.}(1983).
%\newblock {Tail probabilities for finite supercritical braching processes.}
%\newblock {\em Technical Report, Dept. Coputer Science and Engineering, Southern Methodist %University, Dallas TX.}


%\bibitem{MO92}
%{\sc Miller,H.I. and O'Sullivan,J.A.} (1992).
%\newblock {Entropies and conbinatorics of random branching processes and context free languages.}
%\newblock {\em IAnn. Math. Sil.}6:61-64

\bibitem{LJCA06}
{\sc Li J.P. and Chen A.Y.}(2006).
\newblock {Markov branching processes with immigration and resurrection.}
\newblock {\em Markov Process. Related Fields}, 12(1):139-168
%\bibitem{A87}
%{\sc Asmussen,S.} (2003).
%\newblock {\em Applied probability and Queues}.
%\newblock 2nd ed. Springer-Verlag, New York.

%\bibitem{BB96}
%{\sc Bayer,N. and Boxma,O.J.} (1996).
%\newblock  Wiener-Hopf analysis of an M/G/1 queue with
%negative customers and of a related class of random walks.
%\newblock {\em Queueing Systems \/}~ 23, 301-316.


%\bibitem{CT83}
%{\sc Chaudhry,M.L. and Templeton,J.G.C.} (1983).
%\newblock {\em A First Course in
%Bulk Queues}.
%\newblock Wiley: New York.

%\bibitem{CPZL04}
%{\sc Chen,A. Y., Pollett,P. K., Zhang,H. J. and Li,J.
%P.}(2004).
%\newblock The collision branching process.
%\newblock {\em J. Appl. Prob.\/},~41, 1033-1048.

%\bibitem{CR90}
%{\sc Chen,A.Y. and Renshaw,E.} (1990).
%\newblock Markov branching processes with instantaneous immigration.
%\newblock {\em Prob. Theor. Rel. Fields\/}~87, 209-240.

%\bibitem{CR93a}
%{\sc Chen,A.Y. and Renshaw,E.} (1993).
%\newblock Existence and uniqueness criteria for conservative uni-instantaneous denumerable Markov processes
%\newblock {\em Prob. Theor. Rel. Fields\/}~94, 427-456.

%\bibitem{CR93b}
%{\sc Chen, A.Y. and Renshaw, E.} (1993).
%\newblock Recurrence Markov branching processes with immigration.
%\newblock {\em Stoch. Proc. Appl.\/}~45, 231-242.

%\bibitem{CR95}
%{\sc Chen,A.Y. and Renshaw,E.} (1995).
%\newblock Markov branching processes regulated by emigration and large immigration.
%\newblock {\em Stoch. Proc. Appl.\/}~57, 339-359.


\bibitem{Li-Chen08}
{\sc Li J.P. and Chen A.Y.}(2008).
\newblock {Decay property of stopped Markovian Bulk-arriving queues.}
\newblock {\em Adv. Appl. Probab.}, 40(1):95-121.


\bibitem{LCP12}
{\sc Li J.P. and Chen A.Y. and Pakes A.G.}(2012)
\newblock {Asymptotic properties of the Markov branching process with immigration.}
\newblock {\em J. Theoret. Probab.}, 25(1):122-143.

\bibitem{LJL11}
{\sc Li J.P. and Liu Z.M.}(2011).
\newblock {Markov branching processes with immigration-migration and resurrection.}
\newblock {\em Sci. China Math.}, 54(1):1043-1062.


%\bibitem{CR00}
%{\sc Chen,A.Y. and Renshaw,E.} (2000).
%\newblock Existence, recurrence and equilibrium properties  of Markov branching processes with instantaneous immigration.
%\newblock {\em Stoch. Proc. Appl.\/}~88, 177-193.


%\bibitem{LL18}
%{\sc Li,L.Y. and Li,J.P.} (2020).
% \newblock Large deviation rates for supercritical branching processes with immigration.
% \newblock {\em }, manuscript.


\bibitem{LZ16}
{\sc Liu J.N. and Zhang M.}(2016).
\newblock {Large deviation for supercritical branching processes with immigration.}
\newblock {\em Acta Mathematica Sinica, English Series.} 32(8):893-900.


\bibitem{RC97}
{\sc Renshaw E. and Chen A.Y.}(1997).
\newblock {Birth-death processes with mass annihilation and state-dependent immigration.}
\newblock {\em Comm. Statist. Stochastic Models}, 13(2):239-253.



\bibitem{SBA49}
{\sc Sevastyanov B.A.}(1949).
\newblock {\em On certain types of Markov processes}\ (in Russian).
\newblock {\em Uspehi Mat. Nauk}, {\bf 4}, 194.


\bibitem{SZ17}
{\sc Sun Q. and Zhang M.}(2017).
\newblock {Harmonic moments and large deviations for supercritical branching processes with immigration.}
\newblock {\em Frontiers of Mathematics in China}, 12(5):1201--1220.

\bibitem{VVA74}
{\sc Vatutin V.A.}(1974).
\newblock {Asymptotic behavior of the probability of the first degeneration for branching processes with immigration.}
\newblock {\em Teor. Verojatnost. i Primenen}, 19(1):26-35.


\bibitem{YM75}
{\sc Yamazato M.}(1975).
\newblock {Some results on continuous time branching processes with state-dependent immigration.}
\newblock {\em J. Math. Soc. Japan}, 27(3):479-496.




%\bibitem{CR04}
%{\sc Chen,A.Y. and Renshaw,E.} (2004).
%\newblock Markovian Bulk-Arriving Queues with State-dependent Control at Idle Time.
%\newblock {\em Adv. Appl. Prob.\/}~36, 499-524.

%\bibitem{CZH02}
%{\sc Chen,A.Y.,Zhang,H.J.and Hou,Z. T.} (2002).
%\newblock Feller Transition Functions, Resolvent Decomposition Theorems and their Application in Unstable Denumerable Markov Processes.
%\newblock {\em Markov Processes and Controlled Markov Chains.\/}~Ed. by Hou Z.T., Filar J.A. and Chen A.Y. ~Ch 2, 15-40.

%\bibitem{CZH02}
%{\sc Chen,A.Y.,Zhang,H.J.and Hou,Z. T.} (2002).
%\newblock .
%\newblock {\em Stoch. Proc. Appl.\/}~88, 177-193.

%\bibitem{C67}
%{\sc Chung,K.L.} (1967).
%\newblock {\em Markov Chains with Stationary Transition
%Probabilities}.
%\newblock 2nd edn. Springer, New York.

%\bibitem{DS67}
%{\sc Darroch, J.N. and Seneta, E.} (1967).
%\newblock {\em On quasi-stationary distributions in absorbing continuous-time finite Markov chains}.
%\newblock {\em J. Appl. Prob.\/ }, ~ 4,192-196.

%\bibitem{DN99}
%{\sc Dudin,A. and Nishimura,S.} (1999).
%\newblock A BMAP/SM/1 queueing system with Markovian  arrival input of disasters.
%\newblock {\em J. Appl. Prob.\/},~36, 868-881.

%\bibitem{Fla74}
%{\sc Flaspohler, D.C.}(1974)
%\newblock Quasi-stationary distributions for absorbing continuous-time denumerable Markov chains.
%\newblock {\em Ann. Inst. Statist. Math. \/}, ~26, 351-356.

%\bibitem{G91}
%{\sc Gelenbe,E.} (1991).
%\newblock Product from networks with negative and positive customers.
%\newblock {\em J. Appl. Prob.\/},~28, 656-663.

%\bibitem{GGS91}
%{\sc Gelenbe,E., Glynn, P. and Sigman} (1991).
%\newblock Queues with negative arrivals.
%\newblock {\em J. Appl. Prob.\/},~28, 245-250.


%\bibitem{GH85}
%{\sc Gross, D. and Harris, C.M.}(1985).
%\newblock {\em Fundamentals of Queueing Theory}.
%\newblock Wiley: New York


%\bibitem{HP93}
%{\sc Harrison,P.G. and Pitel,E. } (1993).
%\newblock Sojourn times in single-server queues with  negative
%customers.
%\newblock {\em J. Appl. Prob.\/},~30, 943-963.

%\bibitem{H93}
%{\sc Henderson,W. } (1993).
%\newblock Queueing networks with negative customers and negative  queue lengths.
%\newblock {\em J. Appl. Prob.\/},~30, 931-942.

%\bibitem{HG88}
%{\sc Hou,Z.T. and Gou,Q.F. } (1988).
%\newblock {\em Homogeneous Denumerable Markov Processes}.
%\newblock Springer-Verlag: Berlin.

%\bibitem{IT71}
%{\sc It$\hat{o}$, K.} (1971)
%\newblock Poisson point processes attached to Markov processes.
%\newblock {\em Proc. 6th Berkeley Symp. Math. Statist. Prob.\/}, Vol. 3, University of Califoria Press, Berkeley, 225-240.

%\bibitem{JS96}
%{\sc Jain,G. and Sigman,K.} (1996)
%\newblock A Pollaczek--Khintchine formula for M/G/1 queues  with disasters.
%\newblock {\em J. Appl. Prob.\/},~33, 1191-1200.

%\bibitem{KE83}
%{\sc Kelly, F.P.} (1983)
%\newblock Invariant measures and the generator.
%\newblock {\em \/} In Kingman, J.F.C. and Reuter, G.E., eds. {\em Probability, Statistics, and Analysis.} {\em %London Math. Soc. Lecture Notes Series 79.}
% Cambridge University Press,  1983, pp. 143-160.

%\bibitem{Kij93}
%{\sc Kijima, M.} (1963)
%\newblock Quasi-limiting distributions of Markov chains that are skip-free to the left in continuous-time.
%\newblock {\em J. Appl. Prob.\/},~30, 509-517.

%\bibitem{Kin63}
%{\sc Kingman,J.F.C.} (1963)
%\newblock The exponential decay of Markov transition probability.
%\newblock {\em Proc. London Math. Soc.\/},~13, 337-358.

%\bibitem{K75}
%{\sc Kleinrock,I.} (1975)
%\newblock {\em Queueing Systems},
%\newblock Vol.1. Wiley: New York.

%\bibitem{LC07}
%{\sc Li, J.P. and Chen, A.Y.} (2007)
%\newblock Decay Parameter and Quasi-stationary Distributions for Stopped $M^X/M/1$ Queueing
%Models.
%\newblock {\em  \/}, ~, .

%\bibitem{L91}
%{\sc Lucantoni, D. M.} (1991)
%\newblock New results on the single server queue with a batch Markovian arrival process.
%\newblock {\em Stoch. Models \/}, ~7, 1-46.

%\bibitem{LN94}
%{\sc Lucantoni,D. M. and Neuts,M.F.}(1994)
%\newblock  Some steady-state distributions for the $MAP/SM/1$ queue.
%\newblock {\em Stoch. Models \/}, ~10, 575-589.

%\bibitem{M91}
% {\sc Medhi,J.}(1991)
% \newblock {\em Stochastic Models in Queuing Theory.}
% \newblock Academic Press, San Diego, CA.

%\bibitem{NP93}
% {\sc Nair,M.G. and Pollett,P.K.}(1993)
% \newblock {\em }On the relationship between $\mu$-invariant measures and quasistationary distributions for %continuous-time Markov chains.
% \newblock {\em Adv. Appl. Probab. \/}, ~25, 82-102.

%\bibitem{N79}
%{\sc Neuts,M.F.}(1979)
%\newblock A versatile Markovian point process.
%\newblock {\em J. Appl. Prob.\/},~16, 764-779.

%\bibitem{N81}
%{\sc Neuts, M. F.}(1981)
%\newblock {\em Matrix-Geometric Solution in Stochastic Models: An Algorithmic Approach.}
% \newblock Johns Hopkins University Press, Baltimore, MD.

%\bibitem{NS97}
%{\sc Nishimura,S. and Sato,H.}(1997)
%\newblock  Eigenvalue expression for a batch Markovian arrival process.
%\newblock {\em J. Operat. Res. Soc. Japan \/},~ 40, 122-132.

%\bibitem{P88}
%{\sc Pollett,P.K.}(1988)
%\newblock  Reversibility, invariance and mu-invariance.
%\newblock {\em Adv. Appl. Prob. \/},~ 20, 600-621.

%\bibitem{P95}
%{\sc Pollett,P.K.}(1995)
%\newblock  The determeination of quasi-instationary distribution
%directly from the transition rates of an absorbing Markov chain.
%\newblock {\em Mathl. Comput. Modelling \/},~ 22, 279-287.

%\bibitem{P99}
%{\sc Pollett,P.K.}(1999)
%\newblock Quasistationary distributions for continuous time Markov chains when absorption is not certain.
%\newblock {\em J. Appl. Prob. \/},~ 36, 268-272.

%\bibitem{PK91}
%{\sc Parthasarathy,P.R. and Krishna Kumar,B.} (1991)
%\newblock Density-dependent birth and death processes with state-dependent immigration.
%\newblock {\em Math. Comput. Modelling \/}, ~15, 11-16.

%\bibitem{RC97}
%{\sc Renshaw,E. and Chen,A.} (1997)
%\newblock Birth-death processes with mass annihilation and
%state-dependent immigration.
%\newblock {\em Stochastic Models \/},~13, 239-254.

%\bibitem{RW90}
%{\sc Rogers,L.C.G. and Williams,D.} (1990)
%\newblock {\em Diffusions, markov Processes and Martingales.}
%\newblock {\em \/}, Vol.2, Wiley, Chichester.

%\bibitem{S89}
%{\sc Stadje,W.}(1989)
%\newblock Some exact expressions for the bulk-arrival queue $M^X/M/1.$
%\newblock {\em Queuing Systems \/}, ~4, 85-92.

%\bibitem{TW74}
%{\sc Tweedie, R.L.}(1974)
%\newblock Some ergodic properties of the Feller minimal process.
%\newblock {\em Quart, J. Math. Oxford \/}, ~(2)25, 485-493.

%\bibitem{Van85}
%{\sc Van Doorn, E.A.}(1985)
%\newblock Conditions for exponential ergodicity and bounds for the decay parameter of a birth-death process.
%\newblock {\em Adv. appl. Prob. \/}, ~17, 514-530.

%\bibitem{Van91}
%{\sc Van Doorn, E.A.}(1991)
%\newblock Quasi-stationary distributions and convergence to quasi-stationarity of birth-death processes.
%\newblock {\em Adv. appl. Prob. \/}, ~23, 683-700.

%\bibitem{VJ62}
%{\sc Vere-Jones, D.}(1962)
%\newblock Geometric ergidicity in denumerable Markov chains.
%\newblock {\em Quart, J. Math. Oxford \/}, ~(2)13, 7-28.

%\bibitem{Yang90}
%{\sc Yang,X.Q.} (1990).
%\newblock {\em The Construction Theory of Denumerable {M}arkov Processes}.
%\newblock Wiley, New York.

\end{thebibliography}
\end{document}